\documentclass{amsart}

\usepackage{lineno}
\linenumbers

\theoremstyle{plain}
\newtheorem{theorem}{Theorem}[section]
\newtheorem{lemma}[theorem]{Lemma}
\newtheorem{corollary}[theorem]{Corollary}
\newtheorem{proposition}[theorem]{Proposition}
\newtheorem*{corollary*}{Corollary}
\theoremstyle{remark}
\newtheorem*{remark}{Remark}
\DeclareMathOperator{\esssup}{ess\,sup}

\title{$\omega$-recurrence in skew products}

\author{Jon Chaika}
\address{Department of Mathematics\\5734 S. University Avenue\\Chicago IL, 60637}
\email{jonchaika@math.uchicago.edu}

\author{David Ralston}
\address{Depatment of Mathematics/CIS \\ SUNY College at Old Westbury \\ PO 210 \\ Old Westbury, NY 11568}
\email{ralstond@oldwestbury.edu}

\thanks{The first author is supported by NSF grant \#1004372.  Both authors are indebted to the referee of an earlier version for helpful comments on improving clarity and presentation.}
\date{\today}

\begin{document}
\begin{abstract}
The rate of recurrence to measurable subsets in a conservative, ergodic infinite-measure preserving system is quantified by generic divergence or convergence of certain sums given by a function $\omega(n)$.  In the context of skew products over transformations of a probability space, we relate this notion to the more frequently studied question of the growth rate of ergodic sums (including Lyapunov exponents).  We study in particular skew products over an irrational rotation given by bounded variation $\mathbb{Z}$-valued functions: first the generic situation is studied and recurrence quantified, and then certain specific skew products over rotations are shown to violate this generic rate of recurrence.
\end{abstract}
\maketitle

\section{Introduction}

In the study of transformations $T:X \rightarrow X$ which are ergodic with respect to a probability measure $\mu$, the rate of recurrence to a set of positive measure is completely governed by the Birkhoff ergodic theorem: if $A \subset X$ is of measure $\mu(A)=\alpha >0$, then for almost every $x \in X$ we have that the sequence $\{n : T^n(x) \in A\}$ is of \textit{density} $\alpha$ (we write $\#A$ for the number of elements in a set $A$):
\begin{equation}\label{eqn - density} \lim_{i \rightarrow \infty} \frac{\# \{n=0,1,\ldots,i-1 : T^n(x) \in A\} }{i} = \alpha.\end{equation}

If, however, $\mu(X)=\infty$ and $T$ is conservative and ergodic, then for any set $A$ of measure $0 < \mu(A) < \infty$, for almost every $x$ the sequence $\{n : T^n(x) \in A\}$ is infinite but the limit in \eqref{eqn - density} is zero.  To quantify a generic rate of recurrence in this setting, we follow \cite{krengel} and let $\omega: \mathbb{R}^+ \rightarrow \mathbb{R}^+$ be nonincreasing and regularly varying ($\omega(kx) \asymp \omega(x)$ for all $k \in \mathbb{R}^+$). For each $x \in X$ we define
\[\zeta_j^k(x) = \sum_{i=j}^{k-1} \chi_A(T^i x) \omega(i), \quad \zeta(x) = \lim_{n \rightarrow \infty} \zeta_0^n (x)\]
and through the assumption that $T$ is ergodic and conservative we have that $\zeta(x)$ either converges almost everywhere or diverges almost everywhere, and furthermore this convergence or divergence does not depend on the choice of the set $A$ of positive, finite measure.  If $\zeta(x)=\infty$ for almost every $x$, the system is said to be $\omega$-recurrent, while if $\zeta(x)<\infty$ for almost every $x$ the system is called $\omega$-nonrecurrent.

In this work we will concern ourselves with the following situation: let $\{Y, \nu\}$ be a probability space and $S: Y \rightarrow Y$ an ergodic measure-preserving transformation.  Let $f:Y \rightarrow G$ be a function into a countable discrete group $G$, and denote the identity element of $G$ by $e$.  Let $X = Y \times G$ and $\mu$ be the product of $\nu$ and the counting measure on $G$.  Define the \textit{skew transformation}
\begin{equation}\label{eqn - skew}T(y,g) = (S(y), g\cdot f(y)),\end{equation}
and assume that $\{X, \mu, T\}$ is ergodic.  Then the sequence of \emph{ergodic products of $y$} is given by
\[e, f(y), f(y)\cdot f(S(y)), f(y) \cdot f(S(y)) \cdot f(S^2(y)), \ldots, \prod_{i=0}^{n-1} f(S^i (y)), \ldots\]
We then have the cocycle identity for $n,m \geq 0$
\begin{equation}\label{eqn - cocycle}\prod_{i=0}^{n+m-1}f(S^i y)= \prod_{i=0}^{n-1}f(S^i y) \cdot \prod_{i=0}^{m-1}f(S^{n+i}y).\end{equation}

In \S\ref{section - outline process} we will introduce notation and the central result of this work (Theorem \ref{theorem - main abstract theorem}).  This machinery will be applied in \S\ref{section - rotation cocycles} to skew products over irrational rotations; let $Y=\mathbb{R}/\mathbb{Z}$ and $S$ be rotation by $\alpha \notin \mathbb{Q}$, and let $f:Y \rightarrow \mathbb{Z}$ be of bounded variation.  In the particular case that
\begin{equation}\label{eqn - staircase function}f(y) = \chi_{[0,1/2)}(y) - \chi_{[1/2,1)}(y),\end{equation}
this skew product is called the \textit{infinite staircase}.

We will study in \S\ref{section - rotation cocycles} the question of \textit{generic} $\omega$-recurrence in skew products into $\mathbb{Z}$ over irrational rotations: for almost-every choice of $\alpha$, any such ergodic skew product with fixed $f$ of bounded variation is $1 / n$-recurrent (Theorem \ref{theorem - 1 over n}).  We will also apply our results to a particular class of \emph{interval exchange transformations} as another example of how the results of \S\ref{section - outline process} may be applied in general.  Furthermore, using techniques developed in \cite{substitutions1}, we will show in \S\ref{section - staircase} that for any $\omega(n) \in o(n^{\epsilon})$ with $\epsilon < -1/2$ there is an uncountable set of $\alpha$ for which the infinite staircase is \emph{not} $\omega$-recurrent (Theorem \ref{theorem - not omega recurrent for power more than half}).  Finally, the proof of a technical lemma in the theory of continued fractions (Lemma \ref{lemma - technical lemma}) is given in \S\ref{section - lemmas}.

\section{$\omega$-recurrence, ergodic sums and Lyapunov exponents}\label{section - outline process}

Let $\{Y, \nu\}$ be a probability space and $S:Y \rightarrow Y$ be an ergodic measure-preserving transformation.  Let $f:Y \rightarrow G$ be a function into a countable discrete group $G$ with identity element $e$.  Let $\{N_k\}$ be an increasing sequence of positive integers for $k=1,2,\ldots$, and for each $y \in Y$ let
\[r_k(y) = \#\left\{g \in G : g=\prod_{i=0}^{n-1}f(S^i y), \quad n =1,2,\ldots,N_k \right\}\]
Fixing some $\epsilon_1 \in (0,1]$, let $\rho_k(\epsilon_1)=\rho_k$ be such that
\[\nu \left\{ y \in Y : r_k(y) \leq \rho_k \right\} \geq \epsilon_1,\] and then denote
\begin{equation}\label{eqn - sets A} A_k = \left\{ y \in Y : r_k(y) \leq \rho_k\right\}.\end{equation}
Fixing some choice of $k$, let $y \in A_k$ be arbitrary (but fixed), and let $\epsilon_2 \in (0,1)$.  Then a particular $g \in G$ will be called $\epsilon_2$-\textit{crowded} if
\[ \sum_{n=1}^{N_k} \chi_g \left( \prod_{i=0}^{n-1}f(S^i y) \right) \geq \epsilon_2 \frac{N_k}{\rho_k}.\]
Define $G'_k(y, \epsilon_1, \epsilon_2)=G'_k(y)$ to be the set of all $\epsilon_2$-crowded values of $y$.  Then the following is immediate:

\begin{lemma}\label{lemma - lots of balls in crowded bins}
For each $k$ and any $y \in A_k$,
\[ \# \left\{n = 1,2, \ldots, N_k : \prod_{i=0}^{n-1}f(S^i(y)) \in G'_k(y) \right\} \geq (1- \epsilon_2)N_k.\]
\end{lemma}

Let $\epsilon_3 \in (0,1)$.  With $k$ and $y \in A_k$ fixed, $m \in \{1,\ldots, N_k\}$ will be called $\epsilon_3$-\textit{predicting} if
\[ \# \left\{n = 1,2,\ldots,N_k : \prod_{i=0}^{m+n}f(S^i y) = \prod_{i=0}^m f(S^i y)\right\} \geq \epsilon_2 \epsilon_3 \frac{N_k}{\rho_k}.\]
Define $N'_k(y, \epsilon_1,\epsilon_2,\epsilon_3)=N'_k(y)$ to be the set of $\epsilon_3$-predicting times.  

\begin{lemma}\label{lemma - how many predicting}
For each $k$ and any $y \in A_k$, $\#N'_k(y) \geq (1-\epsilon_2) (1-\epsilon_3) N_k - \rho_k$.
\begin{proof}
Fix $y$ and enumerate $G'_k(y)=\{g_1, g_2, \ldots, g_m\}$, where $m \leq \rho_k$.  Let $j_i$ be the number of ergodic products with value $g_i$, so that by Lemma \ref{lemma - lots of balls in crowded bins} we have $j_1 + \ldots +j_m \geq (1-\epsilon_2)N_k$.  For each $i$, then, at least $[(1-\epsilon_3)j_i] = (1-\epsilon_3)j_i - \{(1-\epsilon_3)j_i\}$ times are $\epsilon_3$-predicting, where $[x]$ denotes the integer part of $x$ and $\{x\}=x-[x]$.  Then the number of $\epsilon_3$-predicting times is at least
\[\sum_{i=1}^{m} [(1-\epsilon_3)j_i] \geq (1-\epsilon_3)(1-\epsilon_2)N_k - \sum_{i=1}^m \{(1-\epsilon_3)j_i\} > (1-\epsilon_3)(1-\epsilon_2)N_k-\rho_k. \qedhere\]
\end{proof}
\end{lemma}

Note that from \eqref{eqn - cocycle} it follows that if $y \in Y$, $\tau$, $m \in \mathbb{N}$,
\[\left(\prod_{i=0}^{\tau-1}f(S^i y) = \prod_{i=0}^{\tau +m -1}f(S^i y) \right) \quad \Longrightarrow \left( \prod_{i=0}^{m-1}f\left(S^i (S^\tau y)\right) = e\right).\]

Finally, define $B_k(\epsilon_1, \epsilon_2, \epsilon_3)=B_k$ by
\begin{equation}\label{eqn - define B_k}B_k = \left\{ y \in Y : \sum_{m=1}^{N_k} \chi_{e}\left(\prod_{i=0}^{m-1}f(S^i y)\right) \geq \epsilon_2 \epsilon_3 (N_k/\rho_k)\right\},\end{equation} the set of points which have at least $\epsilon_2 \epsilon_3 N_k/\rho_k$ ergodic products equal to the identity in the first $N_k$ times.

\begin{proposition}\label{proposition - the good B_k}
If $\rho_k \in o(N_k)$, then
\[\liminf_{k \rightarrow \infty}\nu(B_k) \geq \epsilon_1(1 - \epsilon_2)(1-\epsilon_3) > 0.\]
\begin{proof}
Recall the sets $A_k$ from \eqref{eqn - sets A}, and let $C_i = S^{i}(A_k)$, $\tilde{C}_i = \left( C_i \cap B_k \right)$ for $i=0,1,\ldots,N_k-1$.  Combining lemma \ref{lemma - how many predicting}, $\nu(C_0) = \nu(A_k) \geq \epsilon_1$ and the fact that $S$ preserves $\nu$, we obtain
\[\sum_{i=0}^{N_k-1}\nu(\tilde{C}_i) \geq \epsilon_1\left((1-\epsilon_2)(1-\epsilon_3)N_k-\rho_k\right),\]
so the measure of the union 
\[ \bigcup_{i=0}^{N_k-1} \tilde{C}_i\] is at least this quantity divided by the number of sets, $N_k$.  Finally, we trivially note that this union of the $\tilde{C}_i$ is contained within $B_k$, and the conclusion follows from the assumption that $\rho_k \in o (N_k)$.
\end{proof}
\end{proposition}

\begin{lemma}\label{lemma - new lemma}
Suppose that $i_1< i_2 < \ldots < i_m$ and let
\[ y \in \bigcap_{j=1}^m B_{i_j},\] with $B_k$ defined as in \eqref{eqn - define B_k}.  Assume that there is some $\delta>0$ so that for $j=1,2,\ldots,m$ we have
\[ \frac{N_{i_j}}{\rho_{i_j}} \geq \delta \sum_{\ell =1}^{j-1} \frac{N_{i_{\ell}}}{\rho_{i_{\ell}}}.\]
Then if we set $\delta' = \delta/(1+\delta)$,
\begin{equation}\label{eqn - estimate for main theorem} \zeta_1^{N_{i_m}}(y) \geq \delta' \epsilon_2 \epsilon_3 \sum_{j=1}^{m} \omega(N_{i_j}) \frac{N_{i_j}}{\rho_{i_j}} .\end{equation}
\begin{proof}
The quantity $\delta'$ satisfies for all $j=1,2,\ldots,m$
\begin{equation}\label{eqn - alternate delta} \frac{N_{i_j}}{\rho_{i_j}} - \delta' \sum_{\ell=1}^j \frac{N_{i_{\ell}}}{\rho_{i_{\ell}}} \geq \delta' \frac{N_{i_j}}{\rho_{i_j}}.\end{equation}
For any $y \in B'$, first note that as $y \in B_{i_1}$, there are at least $\epsilon_2 \epsilon_3 (N_{i_1}/\rho_{i_1})$ ergodic products equal to the identity before time $N_{i_1}$, so using the fact that $\omega$ is non-increasing and $\delta'<1$,
\[\zeta_1^{N_{i_1}}(y) \geq \delta' \epsilon_2 \epsilon_3 \omega(N_{i_1}) \frac{N_{i_1}}{\rho_{i_1}}.\]
Next, as $y \in B_{i_2}$, there are at least $\epsilon_2 \epsilon_3 (N_{i_2}/\rho_{i_2})$ such times before $N_{i_2}$.  At least $\epsilon_2 \epsilon_3 (N_{i_1}/\rho_{i_1})$ of these times occur before $N_{i_1}$ - perhaps even all of them do.  Again using that $\omega$ is non-increasing, however, we may elect to only replace $\delta' \epsilon_2 \epsilon_3 (N_{i_1}/\rho_{i_1})$ of these times with the smaller value $\omega(N_{i_1})$, and then replace the rest with the even smaller $\omega(N_{i_2})$:
\[ \zeta_1^{N_{i_2}}(y) \geq \epsilon_2 \epsilon_3 \left( \omega(N_{i_1}) \frac{\delta' N_{i_1}}{\rho_{i_2}} + \omega(N_{i_2}) \left( \frac{N_{i_2}}{\rho_{i_2}} - \frac{\delta' N_{i_1}}{\rho_{i_1}}\right) \right) \geq \delta' \epsilon_2 \epsilon_3 \sum_{j=1}^2 \omega(N_{i_j}) \frac{N_{i_j}}{\rho_{i_j}},\]
where the last inequality follows from \eqref{eqn - alternate delta}.  Proceeding in this manner we obtain \eqref{eqn - estimate for main theorem}.
\end{proof}
\end{lemma}

\begin{theorem}\label{theorem - main abstract theorem}
Suppose that $\rho_k \in o(N_k)$, and that there is a $\delta>0$ such that the inequality
\begin{equation} \label{eqn - strange equation at the heart of it all}
\frac{N_k}{\rho_k} \geq \delta \sum_{i=1}^{k-1} \frac{N_i}{\rho_i} \end{equation}
holds for sufficiently large $k$, and assume that $\{X, \mu, T\}$ is conservative and ergodic.  Then the system is $\omega$-recurrent for all $\omega$ such that
\[\sum_{k=1}^{\infty} \omega(N_k) \frac{N_k}{\rho_k} = \infty.\]
\begin{proof}
Assume without loss of generality that \eqref{eqn - strange equation at the heart of it all} holds for all $k$, and again set $\delta' = \delta/(1+\delta)$.  Recall the sets $B_k$ defined in \eqref{eqn - define B_k}, and suppose that $i_1<i_2< \ldots i_m$ and $j_1<j_2< \ldots j_M$ with $j_M \geq i_m$, so that that we may find
\[ B \subset \bigcap_{\ell =1}^m B_{i_{\ell}}, \quad B' \subset \bigcap_{\ell = 1}^M B_{j_{\ell}}.\] Assume that $\mu(B)=\mu(B')$ and $B \cap B' = \emptyset$.  Denote by $t_1 < t_2 < \ldots < t_{\tau}$ the enumeration of $\{i_1, i_2, \ldots i_m\} \cup \{j_1, j_2, \ldots, j_M\}$ in increasing order.

Then using the same technique for estimating $\zeta_1^{N_k}$ from Lemma \ref{lemma - new lemma}, we find that
\begin{align*} \int_{B \cup B'} \zeta_1^{N_{j_M}}(x) d \mu &\geq \int_B \zeta_1^{N_{j_m}}(x) d \mu + \int_{B'}\zeta_1^{N_{j_M}}(x) d \mu\\
& \geq \delta' \epsilon_2 \epsilon_3 \mu(B) \left( \sum_{\ell =1}^{m} \omega(N_{i_{\ell}}) \frac{N_{i_{\ell}}}{\rho_{i_{\ell}}} + \sum_{\ell = 1}^M \omega(N_{j_{\ell}})\frac{N_{j_{\ell}}}{\rho_{j_{\ell}}}\right)\\
&\geq \delta' \epsilon_2 \epsilon_3 \mu(B) \sum_{\ell = 1}^{\tau} \omega(N_{t_{\ell}}) \frac{N_{t_{\ell}}}{\rho_{t_{\ell}}}.
\end{align*}
From Proposition \ref{proposition - the good B_k} we know that $\liminf \nu(B_k)=\beta>0$, so let $D \subset Y$ be an arbitrary set such that $\nu(D)>1-\beta/2$.  Then for sufficiently large $k$ we have
\[ \nu \left( D \cap B_k \right) > \frac{\beta}{2}.\]
Without loss of generality assume that the above inequality holds for all $k$.  Assume further that the various $B_i \cap D$ coincide; the above computation shows that our estimate for this case is no larger than the general case. 
\[ \int_D \zeta_0^{N_k}(y) d\nu \geq \delta' \epsilon_2 \epsilon_3 \frac{\beta}{2} \sum_{ \ell=1}^k \omega(N_{i_{\ell}}) \frac{N_{i_{\ell}}}{\rho_{i_{\ell}}}.\]

It follows from our assumptions, then, that for \emph{any} set $D \subset Y$ of measure $\nu(D) > 1- \beta/2$, we have 
\[\int_{D \times\{e\}} \zeta(y,g) d\mu = \int_D \zeta(y) d \nu =\infty,\]
which implies that $\zeta(y,g)=\infty$ on a set of positive $\mu$-measure.
\end{proof}
\end{theorem}

If $G$ is a discrete metric group, we will say that $G$ is \emph{at most $d$-dimensional} if 
\[ \#\left\{ g \in G : \|g \| \leq n\right\} \in O(n^d).\]

\begin{corollary}\label{corollary - lyapunov}
Assume that $\{Y, \nu, S\}$ is ergodic, with $f:Y \rightarrow G$, where $G$ is discrete and at most $d$-dimensional.  Denote
\begin{equation}\label{eqn - lyapunov} \lambda= \limsup_{n \rightarrow \infty} \frac{\log\left( \esssup_{y \in Y} \| \sum_{i=0}^{n-1}f(S^i y)\| \right)}{\log n},\end{equation} the principal Lyapunov exponent.  If $ 1/d > \lambda$ and the skew product is ergodic (recall \eqref{eqn - skew}), then for any $1/d>\lambda'>\lambda$, the skew product is $n^{(d \lambda'-1)}$-recurrent.
\begin{proof}
Between the dimension $d$ and the principal Lyapunov exponent $\lambda$, we may set $N_k = \gamma^k$ for some $\gamma>1$ and (for sufficiently large $k$) $\rho_k = \gamma^{kd\lambda'}$.  We apply Theorem \ref{theorem - main abstract theorem}; the verification of \eqref{eqn - strange equation at the heart of it all} is direct.
\end{proof}
\end{corollary}

\section{$\omega$-recurrence in skew products into $\mathbb{Z}$ over rotations}\label{section - rotation cocycles}

We now set $Y=\mathbb{R}/ \mathbb{Z}$, $\nu$ as Lebesgue measure, $S(y) = y + \alpha \mod 1$ for some $\alpha \in (0,1)\setminus \mathbb{Q}$, and $f:Y \rightarrow \mathbb{Z}$ to be of bounded variation.  We use \textit{standard continued fraction} notation, where:
\[\alpha = [a_1,a_2,\ldots] = \cfrac{1}{a_1+\cfrac{1}{a_2+\cfrac{1}{\ddots}}}.\]
The $a_i=a_i(\alpha)$ are called \emph{partial quotients} of $\alpha$.  Beginning with $q_0=1$ and $q_1=a_1$, recursively define $q_{n+1}=a_{n+1}q_n+q_{n-1}$, and let
\[\textbf{a}_n(\alpha)=\textbf{a}_n = \sum_{i=1}^n a_i.\]

The \textit{Denjoy-Koksma inequality} states 

\begin{equation}\label{eqn - denjoy} \forall n \in \mathbb{N}, \, \forall y \in [0,1), \quad \left| \sum_{i=0}^{q_n -1} f(y+ i \alpha) - q_n \int_0^1 f(t)dt \right| < \textrm{Var}(f).\end{equation}
The proof of the following Lemma is postponed until \S\ref{section - lemmas}:
\begin{lemma}\label{lemma - technical lemma}
For almost every $\alpha$, there is a $\delta > 0$ such that for all $k$
\[ \frac{q_{k}}{\textbf{a}_{k}} \geq \delta \sum_{i=1}^{k-1} \frac{q_i}{\textbf{a}_i} .\]
\end{lemma}

\begin{theorem}\label{theorem - 1 over n}Let $f$ be a $\mathbb{Z}$-valued function of bounded variation and zero mean such that for almost every $\alpha$ the corresponding skew product over rotation by $\alpha$ is conservative and ergodic. Then for almost every $\alpha$ the corresponding skew product over rotation by $\alpha$ is $(1/n)$-recurrent.
\begin{proof}
Set $N_k=q_k$, and via \eqref{eqn - denjoy} we let $\rho_k = 2\textrm{Var}(f)\textbf{a}_k$, and $\omega(n)=1/n$.  Apply Theorem \ref{theorem - main abstract theorem}, justified by Lemma \ref{lemma - technical lemma}, disregarding the constant $2\textrm{Var}(f)$, and the known fact (see e.g. the final Corollary in \cite{MR1556899}) that for almost-every $\alpha$ we have
\[\sum_{i=1}^{\infty} \frac{1}{\textbf{a}_i} = \infty. \qedhere\]
\end{proof}
\end{theorem}
\begin{remark}
It follows that for almost-every $\alpha$ and any bounded-variation function $f:\mathbb{R}/\mathbb{Z} \rightarrow \mathbb{Z}$, the principal Lyapunov exponent (recall \eqref{eqn - lyapunov}) is zero; this result is already known and follows via \eqref{eqn - denjoy}: see for example \cite[\S2]{conze}.
\end{remark}
\begin{remark} The sequence $\textbf{a}_i(\alpha)/(i \log i)$ converges to $1/\log 2$ in measure (\cite[\S 4]{MR1556899}). So for generic $\alpha$ it is possible to strengthen this recurrence to any $\omega$ of the form (for sufficiently large $n$)
\[\omega(n) = \frac{1}{n \cdot \log^{(3)} n \cdot \log^{(4)} n \cdots \log^{(j)} n},\]
where $\log^{(i)}n$ is the $i$-th iterated logarithm and $j$ is arbitrary; approximating $\textbf{a}_i > i \log i$ for all but a zero-density sequence of $i$ and $q_i > (K-\epsilon)^i$ (where $K$ is the \textit{Khintchine-Levy constant}) for all sufficiently large $i$ and $\epsilon>0$ arbitrary, we see that for some $C > 0$
\[ \omega(q_k)\frac{q_k}{\textbf{a}_k} \geq \frac{C}{k \cdot \log k \cdot \log^{(2)} k \cdots \log^{(j-1)} k}\] along a sequence of $k$ whose complement (in $\mathbb{N}$) is of zero density.
\end{remark}

A natural generalization of a rotation on $\mathbb{R}/\mathbb{Z}$ is an \emph{interval exchange transformation}, or \emph{IET}, defined on finitely many intervals.  We are not particularly concerned with the definition or properties of IETs; an excellent survey on the subject is \cite{MR2219821}.  For $T$ an IET of \emph{periodic type} (again, we are not concerned with the specific definition here), we have for all $n \in \mathbb{N}$ \cite[Theorem 2.2]{conze-fraczek}:
\begin{equation}\label{eqn - periodic type bound} \sup_{x \in [0,1]} \left| \sum_{i=0}^{n-1} f \circ T^i (x) \right| \leq C \cdot \left(\log n\right)^{M+1} \cdot n^{\theta_2/\theta_1} \cdot V(f),\end{equation}
where $f: [0,1) \rightarrow \mathbb{R}$, $V(f)$ is the variation of $f$, $0 \leq \theta_2 < \theta_1$ are the two largest Lyapunov exponents, $M$ is an explicit positive integer not larger than the number of intervals exchanged and $C$ is a constant not dependent on $n$.

\begin{lemma}\label{lemma - technical lemma for IET}
Let $T$ is an interval exchange of periodic type, ergodic with respect to Lebesgue measure, $f: S^1 \rightarrow \mathbb{Z}$ be of bounded variation, and define for some fixed $\gamma>1$
\[N_k = \gamma^k, \quad \rho_k = C k^{M+1} \gamma^{k \frac{\theta_2}{\theta_1}},\]
where the constant $C$ is $2 \textrm{Var} (f)$ times the constant given by \eqref{eqn - periodic type bound} and does not depend on $k$.
Then there is some $\delta>0$ such that \eqref{eqn - strange equation at the heart of it all} holds:
\[ \frac{N_k}{\rho_k} \geq \delta \sum_{j=1}^{k-1} \frac{N_j}{\rho_j}.\]
\begin{proof}
We will show that $C_k$ are bounded, where
\[C_k=\frac{\rho_k}{\gamma^k} \sum_{i=1}^{k-1} \frac{\gamma^i}{\rho_i}.\]
Denote $\epsilon = \gamma^{1-\theta_2 / \theta_1}>1 $, so that 
\begin{align*}
C_k &= \frac{k^{M+1}}{\epsilon^k} \sum_{i=1}^{ [k/2] } \frac{\epsilon^i}{i^{M+1}}+ \frac{k^{M+1}}{\epsilon^k}\sum_{i=[k/2]+1}^{k-1} \frac{\epsilon^i}{i^{M+1}}\\
&\leq \frac{k^{M+1}}{\epsilon^k} \sum_{i=1}^{[k/2]} \epsilon^i + \frac{k^{M+1}}{\epsilon^k \cdot [k/2]^{M+1}} \sum_{i=[k/2]+1}^{k-1} \epsilon^i\\
\end{align*}
If we set $n \geq 1 - \log(\epsilon -1)$, we may bound
\[ \sum_{i=1}^{k} \epsilon^i \leq \epsilon^{k+n},\] so we have (recall $M$ and $n$ are constants)
\begin{align*}
C_k &\leq \frac{k^{M+1} \epsilon ^{k/2+n}}{\epsilon^k} + \frac{2^{M+1} \epsilon^{k+n-1}}{\epsilon^k}\\
& \leq \frac{k^{M+1} \epsilon^{n+1}}{\epsilon^{k/2}} + 2^{M+1} \epsilon^{n-1}\\
\limsup_{k \rightarrow \infty} C_k &\leq 2^{M+1} \epsilon^n. \qedhere
\end{align*}
\end{proof}
\end{lemma}

Then we may obtain a result stronger than that implied solely by Corollary \ref{corollary - lyapunov}:
\begin{theorem}\label{theorem - for IETs}
Let $T$ be is an ergodic IET of periodic type defined on $M$ intervals, with $f:S^1 \rightarrow \mathbb{Z}$ of bounded variation, and assume further that the skew product given by $f$ is ergodic. Then the skew product is $(\log n)^M n^{(\theta_2/\theta_1 -1)}$-recurrent.
\begin{proof}
The proof is direct in light of Theorem \ref{theorem - main abstract theorem}, Lemma \ref{lemma - technical lemma for IET}, and divergence of $\sum 1/(n \log n)$. 
\end{proof}
\end{theorem}

\begin{remark}
It is no more difficult to expand the result of Theorem \ref{theorem - for IETs} to $\omega(n)$ of the form
\[\omega(n) = \frac{\left(\log n\right)^M}{n^{\frac{\theta_1-\theta_2}{\theta_1}} \cdot \log^{(2)}n \cdot \log^{(3)} n \cdots \log^{(j)}n},\]
where $\log ^{(i)}(n)$ represents the $i$-th iterated logarithm, and $\omega(n)$ is defined for sufficiently large $n$.
\end{remark}

\section{Non $\omega$-recurrence in the infinite staircase}\label{section - staircase}

We now turn our attention to the problem of \textit{specific} $\alpha$, and the rates of $\omega$-recurrence in the associated infinite staircase (recall \eqref{eqn - staircase function}).  We will restrict the form of $\alpha$, the rotation on $\mathbb{R}/\mathbb{Z}$:
\begin{equation}\label{eqn - alpha heavy form} \alpha = [2r_1,s_1,2r_2,s_2,\ldots] \quad (r_i, s_i \in \mathbb{Z}^+).\end{equation}
A \textit{substitution} is a homomorphism from the free monoid on a finite set (the \textit{alphabet} of the substitution) to itself, and may be extended to also act on infinite sequences on this set.  Elements of the free monoid are referred to as \textit{words}, and infinite sequences are frequently referred to as \textit{infinite words}.  The \textit{concatenation} of two (finite) words $\omega_1$, $\omega_2$ is their product $\omega_1 \omega_2$, and if $\omega=\omega_1 \omega_2$, then $\omega_1$ is called a \textit{left factor} of $\omega$, and $\omega_2$ is called a \textit{right factor}.  If $\omega'$ is a factor of $\omega$ and $\omega' \neq \omega$, then $\omega'$ is called a \textit{proper factor}.  The space of infinite words is compact in the product topology (the alphabet is finite), and if finite words are equated with open cylinder sets in this space, one may refer in the natural manner to the limit of a sequence of finite words.  In any dynamical system $\{X, \mu, T\}$, we may partition $X$ into a finite number of sets indexed by an alphabet, and the orbit of any point $x \in X$ corresponds to a sequence $\{\omega_0, \omega_1, \ldots\}$ in this alphabet, with $\omega_n$ being given by the symbol corresponding to the partition element containing $T^n(x)$.  This sequence is called the \textit{symbolic encoding} of the orbit of $x$ with respect to the given partition.

Fix the alphabet $\{A,B,C\}$ and define the substitutions $\sigma_i$ to be the homomorphisms determined by
\[ \sigma_i: \left\{ \begin{array}{l}
A \rightarrow A \left( A^{r_i} B^{r_i-1} C \right)\left( A^{r_i} B^{r_i-1} C \right)^{s_i-1} \\
B \rightarrow A \left( A^{r_i -1} B^{r_i} C \right)\left( A^{r_i} B^{r_i-1} C \right)^{s_i-1} \\
C \rightarrow A \left( A^{r_i -1} B^{r_i} C \right)\left( A^{r_i} B^{r_i-1} C \right)^{s_i}
\end{array} \right. \]
For convenience denote $\sigma^{(n)}=\sigma_1 \circ \sigma_2 \circ \cdots \circ \sigma_n$.

\begin{proposition}\label{proposition - all the substitution results}
The symbolic coding of the orbit of $0$ under rotation by $\alpha$, whose continued fraction expansion is of the form \eqref{eqn - alpha heavy form}, with respect to the partition
\[A = \left[0,1/2 \right), \quad B = \left[ 1/2, 1-\alpha \right), \quad C=\left[1-\alpha,1\right), \]
is given by the infinite word
\[W = \lim_{n \rightarrow \infty} \sigma^{(n)}(A).\]
Furthermore, the encoding of \textit{any} point $y \in [0,1)$ may be presented as beginning with the concatenated word $W_1(y) W_2(y)$, where $W_2(y) \in \{\sigma^{(n)}(A),\sigma^{(n)}(B),\sigma^{(n)}(C)\}$, and $W_1(y)$ is a proper right factor (possibly empty) of one of these words.  Finally, the word $\sigma^{(n)}(A)$ is of length $q_{2n}$, and both $\sigma^{(n)}(B)$ and $\sigma^{(n)}(C)$ are of length $q_{2n}+q_{2n-1}$.
\begin{proof}
That the coding of the origin takes the form given is \cite[Thm. 1.1, Prop. 4.3]{substitutions1}.  That the orbit of any $y$ may be realized through the concatenation given follows from \cite[Prop. 4.1]{substitutions1}.  Finally, the lengths of all the words are computed in \cite[Lem. 5.4]{substitutions1}; we let
\[ M_n = \left[ \begin{array}{c c} (2r_n-1)s_n+1 & s_n \\ (2r_n-1)s_n + r_n & s_n +1 \end{array}\right],\]
and then we have
\[ M_n \cdot M_{n-1} \cdots M_1 \left[ \begin{array}{c} 1 \\ 1\end{array}\right] = \left[ \begin{array}{c} \left| \sigma^{(n)}(A) \right|\\ \left|\sigma^{(n)}(B)\right| = \left| \sigma^{(n)}(C) \right| \end{array} \right].\]

First note that $\sigma_1(A)$ is of length $q_2=2r_1s_1+1$, and similarly $\sigma_1(B)$ and $\sigma_1(C)$ are of length $q_2+q_1=2r_1(s_1+1)+1$.  Assume, then, that $\sigma^{(n-1)}(A)$ is of length $q_{2n-2}$ and $\sigma^{(n-1)}(B)$ and $\sigma^{(n-1)}(C)$ are of length $q_{2n-2}+q_{2n-3}$.  Then using the matrix product formula above, we inductively find
\begin{align*} |\sigma^{(n)}(A)| &= ((2r_n-1)s_n+1)q_{2n-2} + s_n \left(q_{2n-2}+q_{2n-3}\right)\\
&=2r_ns_nq_{2n-2}-s_nq_{2n-2}+s_nq_{2n-2}+s_nq_{2n-3}+q_{2n-2}\\
&=s_n(2r_nq_{2n-2}+q_{2n-3})+q_{2n-2}\\
&=s_n(q_{2n-1})+q_{2n-2}=q_{2n}.
\end{align*}
Similarly $|\sigma^{(n)}(B)|=|\sigma^{(n)}(C)|=q_{2n}+q_{2n-1}$.
\end{proof}
\end{proposition}

Given a word $P=p_1p_2\ldots p_n$ of length $n$ in this alphabet, we define
\[g(P) = \#\{i \leq n: p_i=A\} - \#\{i\leq n: p_i = B\} - \#\{i \leq n: p_i=C\}.\]
If $P$ is the coding of the length $(n-1)$ orbit of some point $y$, then $g(P)$ is the $n$-th ergodic sum of $y$.  Further denote for $k \in \mathbb{Z}$
\[g(P,k) = \#\{i \leq n : g(p_1 p_2 \ldots p_i)=k\}.\]

\begin{lemma}\label{lemma - bound on hits to levels for orbit of zero}
Suppose $\alpha$ is of the form given by \eqref{eqn - alpha heavy form} and the $s_i$ are bounded.  Then there is a constant $\tau$ such that for all $n$ we have
\[ \max\left\{g(P,k): k \in \mathbb{Z}, P \in \{\sigma^{(n)}(A),\sigma^{(n)}(B),\sigma^{(n)}(C)\} \right\} \leq \tau q_{2n-2}.\]
\begin{proof}
As the substitutions $\sigma_i$ are homomorphisms, it is direct to show inductively that for our substitutions $\sigma_i$, we have for all $n$
\[g(\sigma^{(n)}(A))=1, \quad g(\sigma^{(n)}(B))=g(\sigma^{(n)}(C))=-1;\] see also \cite[Prop. 5.1]{substitutions1}.  Consider, then, the example of
\begin{align*}
\sigma^{(n)}(A) & = \sigma^{(n-1)}(\sigma_n A)\\
&= \sigma^{(n-1)}(A (A^{r_n}B^{r_n-1}C)^{s_n})\\
&=\sigma^{(n-1)}(A) \left[ \left(\sigma^{(n-1)}(A)\right)^{r_n} \left(\sigma^{(n-1)}(B)\right)^{r_n-1}\left(\sigma^{(n-1)}(C)\right)\right]^{s_n}.
\end{align*}
Using that our sums are an additive cocycle,
\begin{equation}\label{eqn - split}\begin{split}g(\sigma^{(n)}(A),k) = g(\sigma^{(n-1)}(A),k) &+ s_n \left( \sum_{j=1}^{r_n}g(\sigma^{(n-1)}(A),k-j) +\right.\\
& \left. \sum_{j=2}^{r_n+1}g(\sigma^{(n-1)}(B),k-j) + g(\sigma^{(n-1)}(C),k-1)\right).\end{split}\end{equation}
Regardless of how large we choose to make $r_n$, we must have
\[\sum_{j=1}^{r_n}g(\sigma^{(n-1)}(A),k-j) \leq q_{2n-2},\] as $q_{2n-2}$ is the length of the word $\sigma^{(n-1)}(A)$ (Prop. \ref{proposition - all the substitution results}) and each term in the word is accounted for at most once in the sum.  Similarly we have
\[\sum_{j=2}^{r_n+1}g(\sigma^{(n-1)}(B),k-j) \leq q_{2n-2}+q_{2n-3},\quad g(\sigma^{(n-1)}(C),k-1) \leq q_{2n-2}+q_{2n-3}. \]
As the $s_i$ are bounded, we may find $\tau$, independent of $n$, so that via \eqref{eqn - split} $g(\sigma^{(n)}(A),k)  \leq \tau q_{2n-2}$.  Similar arguments apply to $\sigma^{(n)}(B)$ and $\sigma^{(n)}(C)$.
\end{proof}
\end{lemma}

\begin{corollary}\label{corollary - bound on hits to any level any point}
For any $y \in [0,1)$, if we denote $R_n(y)$ to be the symbolic encoding of the $q_{2n}$-length orbit of $y$, then there is a constant $\tau$ (independent of both $y$ and $n$) such that
\[\max\{g(R_n(y),k) : k \in \mathbb{Z}\} \leq \tau q_{2n-2}.\]
\begin{proof}
From Proposition \ref{proposition - all the substitution results}, $R_n(y)$ is a left factor of $W_1(y)W_2(y)$, where $W_2(y) \in \{\sigma^{(n)}(A), \sigma^{(n)}(B), \sigma^{(n)}(C)\}$ is of length at least $q_{2n}$ (and $W_1(y)$ is a proper right factor of one of these words).  Considering the words $W_1(y)$ and $W_2(y)$ independently, we need only double the constant $\tau$ from Lemma \ref{lemma - bound on hits to levels for orbit of zero}.
\end{proof}
\end{corollary}

Note that if the $s_i$ are bounded, then for some $\tau'$ (independent of $n$) we have
\begin{equation}\label{eqn - silly bound on q_n}2r_{n+1}q_{2n}<q_{2n+2}< \tau' r_{n+1}q_{2n}.\end{equation}
\begin{theorem}\label{theorem - not omega recurrent for power more than half}
Let $\omega(n) \in o(1/n^{\epsilon})$ for some $\epsilon>1/2$ and be monotone decreasing and regularly varying, and let $f$ be given by \eqref{eqn - staircase function}.  Then there is an uncountable set of $\alpha$ such that the infinite staircase is \textit{not} $\omega$-recurrent.  In fact, for $A=Y \times\{0\}$ any fixed $\delta>0$, there is an uncountable set of $\alpha$ for which 
\[\zeta(y) \leq \omega(0) + \delta.\]
for \emph{all} $y \in Y$.
\begin{proof}

Let $A = Y \times\{0\}$ and $\delta>0$.  We have via the cocycle identity \eqref{eqn - cocycle}
\[\zeta_0^{q_{2n+2}}(y) \leq \zeta_0^{q_{2n}}(y) + \sum_{l=1}^{2r_{n+1}s_{n+1}+1} \left( \zeta_0^{q_{2n}} (y+(l -1)q_{2n}\alpha \right).\]
By Corollary \ref{corollary - bound on hits to any level any point}, monotonicity of $\omega$, and \eqref{eqn - silly bound on q_n} we have
\begin{equation}\label{eqn - phibound}\zeta_0^{q_{2n+2}}(y) \leq \zeta_0^{q_{2n}}(y) + \sum_{l=1}^{\tau' r_n} \tau q_{2n-2}\omega(l q_{2n}).\end{equation}
By the assumption that $\omega(n) \in o(1/n^{\epsilon})$ and \eqref{eqn - silly bound on q_n}, we have that
\begin{align*}
\sum_{l=1}^{\tau' r_n} \tau q_{2n-2}\omega(l q_{2n}) &\leq \sum_{l=1}^{\tau' r_n} \frac{\tau q_{2n-2}}{l^{\epsilon}(2r_n q_{2n-2})^\epsilon}\\
&< \frac{\tau q_{2n-2}^{1-\epsilon}}{(2r_n)^{\epsilon}} \sum_{l=1}^{\tau' r_n} \frac{1}{l^{\epsilon}}\\
&< \frac{\tau q_{2n-2}^{1-\epsilon}}{(2r_n)^{\epsilon}} \frac{(\tau' r_n)^{1-\epsilon}-\epsilon}{1-\epsilon},
\end{align*}
where the final line follows from elementary calculus (the so-called ``integral comparison" for sums).  As $q_{2n-2}$ does not depend on $r_n$, and $\epsilon>1/2$, with $r_1,s_1,\ldots,r_{n-1},s_{n-1}$ fixed, we are free to set $r_n$ large enough so that
\[\frac{\tau q_{2n-2}^{1-\epsilon} \left((\tau' r_n)^{1-\epsilon}-\epsilon\right)}{(2r_n)^{\epsilon}(1-\epsilon)}< \frac{\delta}{2^n},\]
so letting $n \rightarrow \infty$, via \eqref{eqn - phibound} we have $\zeta(x) \leq \omega(0) + \delta$.  As the choice of $r_n$ is not specific (just some lower bound depending on prior partial quotients), the set of such $\alpha$ that we may construct in this manner is uncountable.
\end{proof}
\end{theorem}

\section{Proof of Lemma \ref{lemma - technical lemma}}\label{section - lemmas}
\begin{proof}
Denote for $k \in \mathbb{Z}^+$
\[C_k = \frac{\textbf{a}_k}{q_k} \sum_{i=1}^{k-1} \frac{q_i}{\textbf{a}_i};\]
we will show that the $C_k$ are bounded for generic $\alpha$.  Then
\begin{align*}
C_{k+1} &= \frac{\textbf{a}_{k+1}}{q_{k+1}} \sum_{i=1}^{k} \frac{q_i}{\textbf{a}_i}\\
&=\frac{\textbf{a}_{k}+ a_{k+1}}{a_{k+1}q_{k}+q_{k-1}} \left(\frac{q_k}{\textbf{a}_k} + \sum_{i=1}^{k-1} \frac{q_i}{\textbf{a}_i}\right)\\
&< \frac{\textbf{a}_{k}+ a_{k+1}}{a_{k+1}q_{k}} \left(\frac{q_k}{\textbf{a}_k} + \sum_{i=1}^{k-1} \frac{q_i}{\textbf{a}_i}\right) \\
&=\left(\frac{1}{a_{k+1}}+\frac{1}{\textbf{a}_k}\right) + \frac{C_k}{a_{k+1}}+\frac{C_k}{\textbf{a}_k}\\
&=\left(\frac{1}{\textbf{a}_k}+\frac{1}{a_{k+1}}\right)(C_k+1)
\end{align*}
For \textit{any} $\alpha$ we have $\textbf{a}_k \geq k$, so for sufficiently large $k$, 
\begin{equation}\label{eqn - contfrac estimate for large quotients}C_k \geq 4, \, a_{k+1} \geq 2 \quad \Longrightarrow \quad C_{k+1} < \frac{2}{3}C_k.\end{equation}
If $\alpha$ only has finitely many $a_i=1$, then clearly from \eqref{eqn - contfrac estimate for large quotients} the $C_k$ remain bounded.  So let $a_{k+i}=1$ for $i=1,2,\ldots,m$.  In this case we have $\textbf{a}_{k+i}=\textbf{a}_k +i$ and $q_{k+i}=\varphi_i q_k + \varphi_{i-1}q_{k-1}$, where $\varphi_i$ is the $i$-th Fibonacci number (beginning with $\varphi_0=\varphi_1=1$).  Then for $m,k>0$
\[C_{k+m} = \frac{\textbf{a}_{k+m}}{q_{k+m}} \sum_{i=1}^{k+m-1}\frac{q_i}{\textbf{a}_i} < \frac{\textbf{a}_k+m}{\varphi_m q_k} \left( \sum_{i=1}^{k-1} \frac{q_i}{\textbf{a}_i} + \sum_{i=0}^{m-1} \frac{\varphi_i q_k + \varphi_{i-1} q_{k-1}}{\textbf{a}_k +i}\right).\]
As $q_{k-1}<q_k$ and $\varphi_{k-1}+\varphi_k=\varphi_{k+1}$, we have
\begin{equation}\label{eqn - intermediary contfrac}C_{k+m} < \frac{\textbf{a}_k + m}{\textbf{a}_k \varphi_m} C_k + \frac{\textbf{a}_k+m}{\textbf{a}_k}\sum_{i=0}^{m-1} \frac{\varphi_{i+1}}{\varphi_m}.\end{equation}

On the one hand, if $m=1$ we have
\[C_{k+1} \leq \frac{k+1}{k}C_k + 2\frac{k+1}{k},\]
which if $a_{k+2} \geq 2$, $C_k \geq 10$, and $k$ is sufficiently large ($k > 44$ suffices), then using \eqref{eqn - contfrac estimate for large quotients}, $C_{k+2} < C_k$.

On the other hand, it is direct to verify for $m \geq 2$ both
\[ \frac{\textbf{a}_k+m}{\textbf{a}_k \varphi_m} \leq 1, \quad \sum_{i=0}^{m-1}  \varphi_i = \varphi_{m+1}-1, \]
so that \eqref{eqn - intermediary contfrac} may be replaced with
\[C_{k+m} < C_k + \left(1+\frac{m}{k} \right) \frac{\varphi_{m+2}}{\varphi_m} < C_k + 3 \left( 1 + \frac{m}{k} \right).\]

Our interest then turns to establishing a reasonable bound on $m/k$ for generic $\alpha$.  We discard the null set of $\alpha$ for which only finitely many $a_i \neq 1$, so for generic $\alpha$ we may define an infinite sequence of $n_i$, $m_i$ such that
\[a_k = 1 \quad \Longleftrightarrow \quad k=n_i+j, \, j\in \{1,2,\ldots,m_i\}.\]
It is an elementary fact in the theory of continued fractions that for generic $\alpha$,
\[ \lim_{k \rightarrow \infty} \frac{\# \{a_j = 1 : j=1,2,\ldots k \} }{k} = \frac{\log 4 - \log 3}{\log 2} < \frac{1}{2}.\]
On the other hand, if there were infinitely many $m_i > n_i$, we would have
\[ \limsup_{k \rightarrow \infty} \frac{\# \{a_j = 1 : j=1,2,\ldots k \}}{k} \geq  \frac{1}{2}.\]
So for generic $\alpha$, we eventually have $m_i \leq n_i$, so for sufficiently large $i$ we have
\begin{equation}\label{eqn - almost done} 0 \leq j \leq m_i \quad \Longrightarrow \quad C_{n_i+j} < C_{n_i}+6.\end{equation}
We then have $a_{n_i+m_i+1} \geq 2$, so using \eqref{eqn - contfrac estimate for large quotients} 

\begin{equation}\label{eqn - last one}C_{n_i} \geq 12 \quad \Longrightarrow \quad C_{n_i+m_i+1} < \frac{2}{3} \left(C_{n_i}+6 \right) \leq C_{n_i}.\end{equation}
Therefore, combining \eqref{eqn - contfrac estimate for large quotients},  \eqref{eqn - almost done}, and \eqref{eqn - last one}, for generic $\alpha$ the sequence $C_k$ remains bounded.
\end{proof}

\bibliographystyle{amsalpha}
\bibliography{omega-recurrence-bifile}
\end{document}